\documentclass[11pt]{article}
\usepackage{amsmath,amsthm,amssymb,tikz,color}

\newtheorem{remark}{Remark}
 \newtheorem{definition}{Definition}
 \newtheorem{conj}[remark]{Conjecture}
 \newtheorem{lemma}[remark]{Lemma}
 \newtheorem{theorem}[remark]{Theorem}
 \newtheorem{proposition}[remark]{Proposition}

\newtheorem{problem}[remark]{Problem}
\newcommand{\ld}{\left|}
\newcommand{\rd}{\right|}
\newcommand{\la}{\lambda}
\newcommand{\li}{{\cal L}}
\newcommand{\s}{{\cal S}}
\newcommand{\x}{{\bf x}}
\newcommand{\uu}{{\bf u}}
\newcommand{\vv}{{\bf v}}
\newcommand{\rk}{{\rm rank}}

\newcommand{\E}{{\cal E}}

\newcommand{\Cay}{\mathrm{Cay}}
\newcommand{\SP}{\mathcal{SP}}
\newcommand{\NC}{\mathcal{NC}}

\newcommand{\gen}[1]{\langle#1\rangle}
\tikzstyle{vertex}=[circle, draw, inner sep=0pt, minimum size=6pt]
\newcommand{\vertex}{\node[vertex]}
\renewcommand{\thefootnote}

\title{Some Algebraic Properties of Sierpi\'nski-Type Graphs}

\author{ M. Farrokhi D. G.$^{\rm 1}$, E. Ghorbani$^{\,\rm 2,}$\thanks{Corresponding author}, H. R. Maimani$^{\,\rm 3}$,  F. Rahimi Mahid$^{\,\rm 3}$\\[.4cm]
{\small\sl $^{\rm 1}$Department of Mathematics, Institute for Advanced Studies in Basic Sciences (IASBS),}\\
{\small\sl and the Center for Research in Basic Sciences and Contemporary Technologies, IASBS,}\\
{\small\sl P. O. Box 45137-66731, Zanjan,  Iran}\\
{\small\sl $^{\rm 2}$Department of Mathematics, K. N. Toosi University of Technology,}\\
{\small\sl P. O. Box 16765-3381, Tehran, Iran}\\
{\small\sl $^{\rm 3}$Department of Basic Sciences, Shahid Rajaee Teacher Training University,}\\
{\small\sl P. O. Box 16783-163, Tehran, Iran}}

\begin{document}
\maketitle
\footnotetext{{\em E-mail Addresses}:  {\tt farrokhi@iasbs.ac.ir, m.farrokhi.d.g@gmail.com} (M. Farrokhi D. G.), {\tt e\_ghorbani@ipm.ir} (E. Ghorbani),
{\tt maimani@ipm.ir} (H. R. Maimani),  {\tt farhad.rahimi@sru.ac.ir} (F. Rahimi Mahid)}

\begin{abstract}
This paper deals with some of the algebraic properties of Sierpi\'nski graphs and a family of regular generalized Sierpi\'nski graphs.
For the family of regular generalized Sierpi\'nski graphs, we obtain their spectrum and characterize those graphs that are Cayley graphs.
As a by-product, a new family of non-Cayley vertex-transitive graphs, and consequently, a new set of non-Cayley numbers are introduced. We also obtain the Laplacian spectrum of Sierpi\'nski graphs in some particular cases, and make a conjecture on the general case.

\vspace{5mm}
\noindent {\bf Keywords:} Sierpi\'nski Graph,  Spectrum,  Laplacian, Cayley graph, Non-Cayley number  \\[.1cm]
\noindent {\bf AMS Mathematics Subject Classification\,(2010):}   05C50, 05C25, 05C75
\end{abstract}

\section{Introduction}
Sierpi\'nski-type graphs show up in a wide range of areas; for instance, physics, dynamical systems, probability, and topology, to name a few.
The Sierpi\'nski gasket graphs form one of the most significant families of such graphs that are obtained by a finite number of iterations that give the Sierpi\'nski gasket in the limit. Several more families of Sierpi\'nski-type graphs have been introduced and studied in the literature (see Barri\`{e}re, Comellas, and Dalf\'o \cite{bcd} and Hinz, Klav\v{z}ar, and Zemlji\v{c} \cite{hkz}). In this paper, we deal with two families of them, as described below.

For positive integers $n,k$, the {\em Sierpi\'nski graph} $S(n, k)$ is defined with vertex set $[k]^n$, where $[k]:=\{1,\ldots,k\}$, and two different vertices $(u_1, \ldots , u_n)$ and $(v_1, \ldots , v_n)$ are adjacent if and only if there exists a $t\in[n]$ such that
\begin{itemize}
  \item $u_i=v_i$ for $i=1,\ldots,t-1$,
  \item $u_t\ne v_t$,
  \item $u_j=v_t$ and $v_j=u_t$ for $j=t+1,\ldots,n$.
\end{itemize}
For instance, the graphs $S(3,3)$ and $S(2,4)$ are depicted in Figure~\ref{fig:S}.
\begin{figure}[h!]
\centering
\begin{tikzpicture}[scale=.65]
\vertex[fill] (a1) at (0,0) [label=below:{\small $222$}]{};
\vertex [fill](a2) at (1.2,0) [label=below:{\small $223$}]{};
\vertex [fill](a3) at (2.4,0) [label=below:{\small $232$}]{};
\vertex [fill](a4) at (3.6,0) [label=below:{\small $233$}]{};
\vertex [fill](a5) at (4.8,0) [label=below:{\small $322$}]{};
\vertex [fill](a6) at (6,0) [label=below:{\small $323$}]{};
\vertex [fill](a7) at (7.2,0) [label=below:{\small $332$}]{};
\vertex [fill](a8) at (8.4,0) [label=below:{\small $333$}]{};
\vertex [fill](b1) at (.6,1.0392304845) [label=left:{\small $221$}]{};
\vertex [fill](b2) at (3,1.0392304845) [label={[shift={(.44,-.3)}]{\small $231$}}]{};
\vertex [fill](b3) at (5.4,1.0392304845) [label={[shift={(-.4,-.3)}]{\small $321$}}]{};
\vertex [fill](b4) at (7.8,1.0392304845) [label= right:{\small $331$}]{};
\vertex[fill] (e1) at (1.2,2.078460969) [label=left:{\small $212$}]{};
\vertex [fill](e2) at (2.4,2.078460969) [label=right:{\small $213$}]{};
\vertex [fill](e3) at (6,2.078460969) [label=left:{\small $312$}]{};
\vertex[fill] (e4) at (7.2,2.078460969) [label=right:{\small $313$}]{};
\vertex [fill](f1) at (1.8,3.1176914535) [label= left:{\small $211$}]{};
\vertex[fill] (f2) at (6.6,3.1176914535) [label=right:{\small $311$}]{};
\vertex [fill](g1) at (2.4,4.156921938) [label=left:{\small $122$}]{};
\vertex[fill] (g2) at (3.6,4.156921938) [label=below:{\small $123$}]{};
\vertex [fill](g3) at (4.8,4.156921938) [label=below:{\small $132$}]{};
\vertex[fill] (g4) at (6,4.156921938) [label= right:{\small $133$}]{};
\vertex [fill](h1) at (3,5.1961524225) [label= left:{\small $121$}]{};
\vertex[fill] (h2) at (5.4,5.1961524225) [label= right:{\small $131$}]{};
\vertex[fill] (i1) at (3.6,6.235382907) [label=left:{\small $112$}]{};
\vertex [fill](i2) at (4.8,6.235382907) [label=right:{\small $113$}]{};
\vertex[fill] (j1) at (4.2,7.2746133915) [label=above:{\small $111$}]{};
\path
(a1) edge (a8)
(j1) edge (a1)
(b1) edge (a2)
(b2) edge (a3)
(b2) edge (a4)
(b3) edge (a5)
(b3) edge (a6)
(b4) edge (a7)
(b4) edge (a8)
(j1) edge (a8)
(g1) edge (g4)
(e1) edge (e2)
(e3) edge (e4)
(e2) edge (b2)
(e3) edge (b3)
(f1) edge (e2)
(f2) edge (e3)
(h1) edge (g2)
(h2) edge (g3)
(i2) edge (i1);
\end{tikzpicture}~~~~
\begin{tikzpicture}
\vertex[fill] (a1) at (0,0) [label=below:$44$]{};
\vertex [fill](a2) at (1.5,0) [label=below:$43$]{};
\vertex [fill](a3) at (3,0) [label=below:$34$]{};
\vertex [fill](a4) at (4.5,0) [label=below:$33$]{};
\vertex [fill](b1) at (0,1.5) [label=left:$41$]{};
\vertex [fill](b2) at (1.5,1.5) [label=above:$42$]{};
\vertex [fill](b3) at (3,1.5) [label=above:$31$]{};
\vertex [fill](b4) at (4.5,1.5) [label=right:$32$]{};
\vertex [fill](e1) at (0,3) [label=left:$14$]{};
\vertex [fill](e2) at (1.5,3) [label= below:$13$]{};
\vertex [fill](e3) at (3,3) [label= below:$24$]{};
\vertex [fill](e4) at (4.5,3) [label= right:$23$]{};
\vertex[fill] (f1) at (0,4.5) [label=left:$11$]{};
\vertex [fill](f2) at (1.5,4.5) [label=above:$12$]{};
\vertex [fill](f3) at (3,4.5) [label=above:$21$]{};
\vertex[fill] (f4) at (4.5,4.5) [label=right:$22$]{};
\path
(a1) edge (a4)
(a1) edge (f1)
(f1) edge (f4)
(f4) edge (a4)
(f1) edge (a4)
(a1) edge (f4)
(b1) edge (a2)
(b1) edge (b2)
(b2) edge (a2)
(b3) edge (a3)
(b3) edge (b4)
(e1) edge (e2)
(e2) edge (f2)
(e1) edge (f2)
(e3) edge (f3)
(f3) edge (e4)
(a3) edge (b4)
(e3) edge (e4);
\end{tikzpicture}
  \caption{The Sierpi\'nski graphs $S(3,3)$ (left) and $S(2,4)$ (right)}\label{fig:S}
\end{figure}
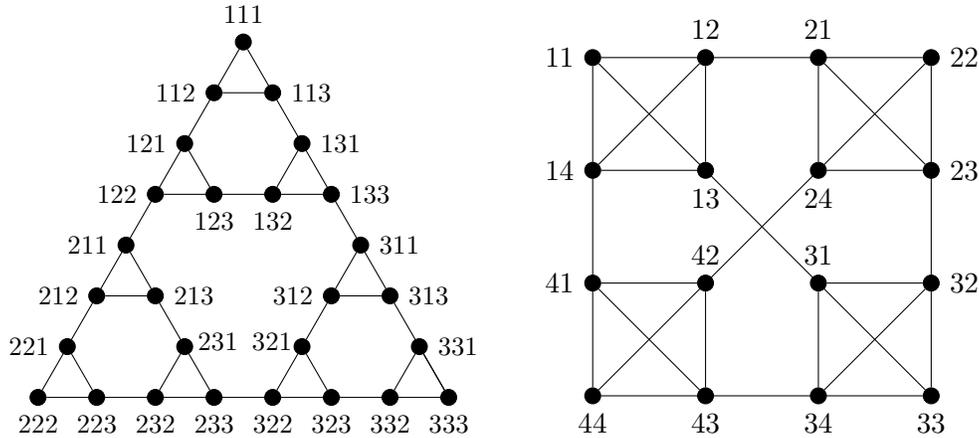
Sierpin\'ski graphs $S(n,k)$ were introduced in Klav\v{z}ar and Milutinovi\'c \cite{kmi}.
The graph $S(n,3)$ is indeed isomorphic to the graph of the Tower
of Hanoi with $n$ disks. The graph $S(n,k)$ has $k^n-k$ vertices of degree $k$ and $k$ vertices of degree $k-1$ that are $(i,\ldots,i)$ for $i\in[k]$.
 These vertices are called the {\em extreme vertices}.

In addition to $S(n,k)$, we consider a `regularization' of them as another
 family of Sierpi\'nski-type graphs.  The graphs $S^{++}(n, k)$, introduced in Klav\v{z}ar and Mohar \cite{km}, are defined as follows. The graph
$S^{++}(1, k)$  is the complete graph $K_{k+1}$. For $n\ge2$, $S^{++}(n, k)$ is the graph obtained
from the disjoint union of $k+1$ copies of $S(n-1, k)$ in which the extreme vertices in
distinct copies of $S(n - 1, k)$ are connected as the complete graph $K_{k+1}$. See Figure~\ref{fig:S++} for an illustration of $S^{++}(3,3)$.
\begin{figure}[h!]
\centering
\begin{tikzpicture}[scale=.55]
\vertex[fill] (a1) at (1.5,0) [label=below:]{};
\vertex [fill](a2) at (9.5,0) [label=below:]{};
\vertex [fill](a3) at (1,0.8660254038) [label=below:]{};
\vertex [fill](a4) at (2,0.8660254038) [label=below:]{};
\vertex [fill](a5) at (9,0.8660254038) [label=below:]{};
\vertex [fill](a6) at (10,0.8660254038) [label=below:]{};
\vertex [fill](a7) at (0.5,1.7320508076) [label=below:]{};
\vertex [fill](a8) at (2.5,1.7320508076) [label=below:]{};
\vertex [fill](b1) at (8.5,1.7320508076) [label=left:]{};
\vertex [fill](b2) at (10.5,1.7320508076) [label= right:]{};
\vertex [fill](b3) at (0,2.5980762114) [label= left:]{};
\vertex [fill](b4) at (1,2.5980762114) [label= right:]{};
\vertex[fill] (e1) at (2,2.5980762114) [label=left:]{};
\vertex [fill](e2) at (3,2.5980762114) [label=right:]{};
\vertex [fill](g1) at (8,2.5980762114) [label=left:]{};
\vertex[fill] (g2) at (9,2.5980762114) [label=below:]{};
\vertex [fill](g3) at (10,2.5980762114) [label=below:]{};
\vertex[fill] (g4) at (11,2.5980762114) [label= right:]{};
\vertex [fill](h1) at (4,3.5980762114) [label= left:]{};
\vertex[fill] (h2) at (5,3.5980762114) [label= right:]{};
\vertex[fill] (i1) at (6,3.5980762114) [label=left:]{};
\vertex [fill](i2) at (7,3.5980762114) [label=right:]{};
\vertex[fill] (j1) at (4.5,4.4641016152) [label=above:]{};
\vertex [fill](p1) at (6.5,4.4641016152) [label=left:]{};
\vertex[fill] (p2) at (5,5.330127019) [label=below:]{};
\vertex [fill](p3) at (6,5.330127019) [label=below:]{};
\vertex[fill] (p4) at (5.5,6.1961524228) [label= right:]{};
\vertex [fill](q1) at (5.5,7.6103659852) [label= left:]{};
\vertex[fill] (q2) at (5,8.476391389) [label= right:]{};
\vertex[fill] (w1) at (6,8.476391389) [label=left:]{};
\vertex [fill](w2) at (4.5,9.3424167928) [label=right:]{};
\vertex [fill](e3) at (6.5,9.3424167928) [label=left:]{};
\vertex[fill] (e4) at (4,10.2084421966) [label=right:]{};
\vertex [fill](f1) at (5,10.2084421966) [label= left:]{};
\vertex[fill] (f2) at (6,10.2084421966) [label=right:]{};
\vertex[fill] (r1) at (7,10.2084421966) [label=above:]{};
\path
(a1) edge (a2)
(a1) edge (a3)
(a1) edge (a4)
(a2) edge (a5)
(a2) edge (a6)
(a3) edge (a7)
(a4) edge (a8)
(a3) edge (a4)
(a5) edge (a6)
(a5) edge (b1)
(a6) edge (b2)
(a7) edge (b3)
(a7) edge (b4)
(b3) edge (b4)
(a8) edge (e1)
(a8) edge (e2)
(e1) edge (e2)
(b4) edge (e1)
(b1) edge (g1)
(b1) edge (g2)
(b2) edge (g3)
(b2) edge (g4)
(g1) edge (g2)
(g2) edge (g3)
(g3) edge (g4)
(e2) edge (h1)
(h1) edge (i2)
(h1) edge (p4)
(h2) edge (j1)
(i1) edge (p1)
(p2) edge (p3)
(p4) edge (q1)
(q1) edge (e4)
(q1) edge (r1)
(e4) edge (r1)
(q2) edge (w1)
(w2) edge (f1)
(e3) edge (f2)
(g1) edge (i2)
(b3) edge (e4)
(g4) edge (r1)
(i2) edge (p4);
\end{tikzpicture}
 \caption{The graph $S^{++}(3,3)$}\label{fig:S++}
\end{figure}
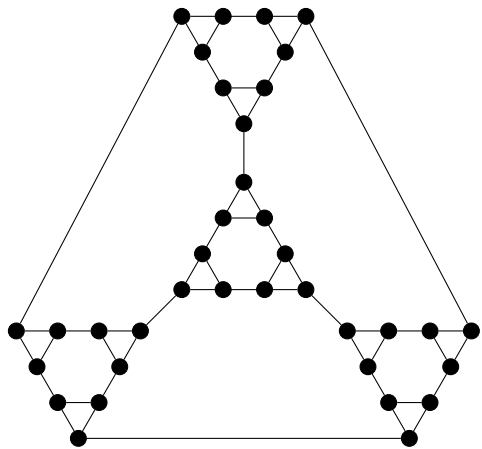

Many properties of Sierpi\'nski-type graphs, including those of $S(n,k)$ and  $S^{++}(n,k)$, have been studied in the literature, for a survey see Hinz, Klav\v{z}ar, and Zemlji\v{c} \cite{hkz}.
In this paper, we investigate some of the algebraic properties of the two families of graphs, namely the spectrum and the property of being a Cayley graph.
More precisely, in Section~\ref{sec:S++}, we determine the spectrum of the graphs $S^{++}(n,k)$.
The Laplacian spectrum of $S(n,k)$ is already known for $k=2,3$. We establish the case $n=2$, in Section~\ref{sec:S},
and make a conjecture on the Laplacian spectrum of $S(n,k)$ in general.
We also characterize the graphs  $S^{++}(n,k)$ that are Cayley graphs in Section~\ref{sec:Cayley}. As a by-product, a new family of non-Calyley vertex-transitive graphs are obtained. From this result, we conclude a new set of square-free non-Cayley numbers in Section~\ref{sec:NC}, and we discuss its distribution.

\section{Spectrum of $S^{++}(n,k)$}\label{sec:S++}

Let $\Gamma$ be a simple graph with vertex set $V(\Gamma)=\{v_1, \ldots ,v_n \}$  and edge set $E(\Gamma)$. Its {\em adjacency matrix} $A(\Gamma)=[a_{ij}]$ is an $n \times n$ symmetric matrix with $a_{ij}=1$ if $v_i$ and $v_j$ adjacent, and $a_{ij}=0$ otherwise.
 The multi-set of the eigenvalues of $A(\Gamma)$ is called the {\em spectrum} of $\Gamma$.

In this section, we determine the spectrum of $S^{++}(n,k)$. As we will see, the recursive structure of these Sierpi\'nski-type graphs also shows up in their spectrum.

We first recall some basic facts.

The {\em incidence matrix} of a graph $\Gamma$ is a 0-1 matrix $X(\Gamma)=[x_{ve}]$, with rows indexed by the vertices and columns
indexed by the edges of $\Gamma$, where  $x_{ve}=1$ if the vertex $v$ is an endpoint of the edge $e$.
For a graph $\Gamma$, $\li(\Gamma)$ denotes the {\em line graph} of $\Gamma$, in which $V(\li(\Gamma))$ corresponds with $E(\Gamma)$, and  two vertices of $\li(\Gamma)$  are adjacent if and only if they have a common vertex as edges of $\Gamma$. The subdivision graph $\s(\Gamma)$ of $\Gamma$ is the graph obtained by inserting a new vertex into every edge of $\Gamma$.
It is easy to verify that
\begin{equation}\label{eq:xtx}X(\Gamma)^\top X(\Gamma)= 2I+A(\li(\Gamma)),\end{equation}
and, moreover, if $\Gamma$ is $k$-regular, then
\begin{equation}\label{eq:xxt}X(\Gamma)X(\Gamma)^\top= kI+A(\Gamma).\end{equation}

The following lemma gives a recursive relation for the graphs $S^{++}(n,k)$.

\begin{lemma}\label{lineiso}
The graph $S^{++}(n+1,k)$ is isomorphic to $\li(\s(S^{++}(n,k)))$.
\end{lemma}
\begin{proof} Let $k$ be fixed.
The graph $\Gamma_n:=S^{++}(n,k)$ can be obtained by the union of $S(n,k)$ and $S(n-1,k)$ by adding a matching between the extreme vertices of the two graphs.
If we consider $\{ 0\}\times[k]^{n-1}$ as the vertex set of $S(n-1,k)$ (to make them compatible with the length $n$ of the vertices of $S(n,k)$),
then $\left(\{0\}\cup[k]\right)\times[k]^{n-1}$ is the vertex set of $\Gamma_n$.
It follows that any edge $e=\{\uu,\vv\}$ of $\Gamma_n$ is of one of the following types:
\begin{itemize}
  \item[(1)] $\uu=(u_1,\ldots,u_r,u,v,\ldots,v)$, $\vv=(u_1,\ldots,u_r,v,u,\ldots,u)$  for some  $r\le n-2$ and $u\ne v$;
  \item[(2)] $\uu=(u_1,\ldots,u_{n-1},u)$, $\vv=(u_1,\ldots,u_{n-1},v)$ with $u\ne v$;
  \item[(3)] $\uu=(0,u,\ldots,u)$, $\vv=(u,u,\ldots,u)$.
\end{itemize}
 Each $e=\{\uu,\vv\}\in E(\Gamma_n)$ is divided into two new edges $e_\uu$ and $e_\vv$  in $\s(\Gamma_n)$, where we assume that $\uu\in e_\uu$ and $\vv\in e_\vv$.
  We define a map $\psi:E(\s(\Gamma_n))\to\left(\{0\}\cup[k]\right)\times[k]^n$
  based on the type of $e$ as follows:
\begin{itemize}
  \item[(i)] If $e$ is of type (1) or (2), then  $\psi(e_\uu)=(\uu,v)$  and $\psi(e_\vv)=(\vv,u)$;
  \item[(ii)] If $e$ is of type (3), then  $\psi(e_\uu)=(\uu,u)$ and $\psi(e_\vv)=(\vv,u)$.
\end{itemize}
It is easily seen that $\psi$ is a one-to-one map.
We show that $\psi$ is an isomorphism from $\li(\s(\Gamma_n))$ to $\Gamma_{n+1}$.
Let $e$ and $e'$ be two edges that share a vertex $\x$ of $\s(\Gamma_n)$.
If $\x=(x_1,\ldots,x_n)$ is an `old' vertex of $\s(\Gamma_n)$, then $\psi(e)=(\x,y)$ and $\psi(e')=(\x,z)$ for some $y\ne z$.
Then, it is clear that $\psi(e)$ and $\psi(e')$ are adjacent in $\Gamma_{n+1}$.
If $\x$ is a `new' vertex of $\s(\Gamma_n)$, then from (i) and (ii), it is clear that $\psi(e)$ and $\psi(e')$ are adjacent in $\Gamma_{n+1}$.
This shows that $\psi$ is indeed a one-to-one homomorphism. As $\li(\s(\Gamma_n))$ and $\Gamma_{n+1}$ have the same number of edges, it follows that $\psi$ is an isomorphism.
\end{proof}

We recall that if $A$ is a non-singular square matrix, then
\begin{equation}\label{eq:det}\left|\begin{array}{cc}A&B\\C&D\end{array}\right|=|A|\cdot\left|D-CA^{-1}B\right|,\end{equation}
where $|\cdot|$ denotes the determinant of a matrix.
Also, recall that if $M$ is a $p\times q$ matrix, then
\begin{equation}\label{eq:MMT}|xI-MM^\top|=x^{p-q}|xI-M^\top M|.\end{equation}
(Note that \eqref{eq:MMT} might not be valid if $p\le q$ and $x = 0$, but this has no effect in our argument since two polynomials that agree in all but finitely many points, agree everywhere.)

Let
\begin{equation}\label{eq:f}f(x)=x^2+(2-k)x-k,\end{equation}
 and let $f^j(x)$ denote the polynomial of degree $2^j$ obtained by $j$ times composition of $f$ with itself.
 As a convention, we let $f^0(x)=x$.

We now give the main result of this section.
\begin{theorem}
Let $k$  be an integer and  $P_n(x)$ denote the characteristic polynomial of the adjacency matrix of $S^{++}(n,k)$. Then, $P_n$ satisfies the recursion relation
\begin{equation}\label{recursion}
P_n(x)=\left(x(x+2)\right)^{k^{n-2}\left({k\choose 2}-1\right)}P_{n-1}(f(x)),~~n\ge2,
\end{equation}
with $P_1(x)=(x-k)(x+1)^k$.
Moreover, for $n\geq 2$, the spectrum of $S^{++}(n,k)$ consists of the following eigenvalues:
\begin{itemize}
  \item[\rm(i)] $k$ with multiplicity $1$,
   \item[\rm(ii)]  the zeros of $f^{n-1}(x)+1$ each with  multiplicity $k$,
  \item[\rm(iii)]  the zeros of $f^j(x)$ each with multiplicity $k^{n-2-j}\left({k\choose 2}-1\right)$ for $j=0,1,\ldots,n-2$,
  \item[\rm(iv)] the zeros of $f^j(x)+2$ each with multiplicity $k^{n-2-j}\left({k\choose 2}-1\right)+1$ for $j=0,1,\ldots,n-2$.
\end{itemize}
\end{theorem}

\begin{proof} Let $\Gamma_n:=S^{++}(n,k)$.
Suppose that $X$ and $Y$ are the incidence matrices of $\Gamma_{n-1}$ and $\s(\Gamma_{n-1})$, respectively.
By Lemma~\ref{lineiso}, $\Gamma_n$ is isomorphic to $\li(\s(\Gamma_{n-1}))$. It follows that
$$YY^\top=\left[\begin{array}{cc}kI_p & X \\ X^\top &   2I_q\end{array}\right],$$
where the matrix is divided according to the partition of the vertices into $p=k^{n-1}+k^{n-2}$ `old' vertices of $\Gamma_{n-1}$
and $q=\frac{1}{2} (k^n+k^{n-1})$ `new' vertices (which have degree $2$) added to $\Gamma_{n-1}$ to obtain $\s(\Gamma_{n-1})$.
Therefore, from \eqref{eq:det},
\begin{align}
    \ld xI-YY^\top\rd&=\ld(x-k)I_p\rd\cdot\ld(x-2)I_q-X^\top\left((x-k)I_p\right)^{-1}X\rd\nonumber\\
    &=(x-k)^p\ld(x-2)I_q-\frac{1}{x-k}X^\top X\rd\nonumber\\
    &=(x-k)^{p-q}\ld(x-2)(x-k)I_q-X^\top X\rd\nonumber\\
    &=(x-k)^{p-q}((x-2)(x-k))^{q-p}\ld(x-2)(x-k)I_p-XX^\top\rd~~~\hbox{(by \eqref{eq:MMT})}\nonumber\\
     &=(x-2)^{q-p}\ld((x-2)(x-k)-k)I_p-A(\Gamma_{n-1})\rd~~~\hbox{(by \eqref{eq:xxt})}\nonumber\\
     &=(x-2)^{q-p}P_{n-1}((x-2)(x-k)-k).\label{eq:Pn-1}
\end{align}
On the other hand, by \eqref{eq:xtx} and \eqref{eq:xxt}, we have
$$P_n(x)=\ld (x+2)I_{2q}-Y^\top Y\rd=(x+2)^{q-p}\ld (x+2)I_{p+q}-YY^\top\rd.$$
Now, from \eqref{eq:Pn-1} it follows that
$$P_n(x)=(x(x+2))^{q-p}P_{n-1}(x(x+2-k)-k),$$
implying \eqref{recursion}.

To prove the second part of the theorem, note that as $\Gamma_1=K_{k+1}$, we have $P_1(x)=(x-k)(x+1)^k$. From \eqref{recursion}, we conclude that
$$P_2(x)=(x(x+2))^{{k\choose2}-1}(f(x)-k)(f(x)+1)^k,$$
and since
\begin{equation}\label{f-k}
f(x)-k=(x+2)(x-k),
\end{equation}
 the assertion follows for $n=2$.
Now assume that $n\ge3$ and the assertion holds for $n-1$. So, we have
$$P_{n-1}(x)=(x-k)\left(f^{n-2}(x)+1\right)^k\prod_{j=0}^{n-3}\left(f^j(x)\right)^{m_{n-3-j}}\left(f^j(x)+2\right)^{1+m_{n-3-j}},$$
in which $m_i=k^i\left({k\choose 2}-1\right)$.
It follows that
$$P_{n-1}(f(x))=(f(x)-k)\left(f^{n-1}(x)+1\right)^k\prod_{j=1}^{n-2}\left(f^j(x)\right)^{m_{n-2-j}}\left(f^j(x)+2\right)^{1+m_{n-2-j}}.$$
This, together with \eqref{recursion} and \eqref{f-k},  implies the result.
\end{proof}

\begin{remark}\rm It is straightforward to see that the zeros of $f^j(x)$ and  $f^j(x)+2$ for
	$j=1$ are $\frac12(k-2\pm\sqrt{k^{2}+4})$ and  $\frac12(k-2\pm\sqrt{k^{2}-4}),$
			 respectively, and for $j\ge2$ are of the form
$$\frac{1}{2}(k-2)\pm\frac{1}{2}\sqrt{k(k+2) \pm 2\sqrt{k(k+2) \pm2\sqrt{ \cdots \pm 2\sqrt{k^2+4}}}}$$
and
$$\frac{1}{2}(k-2)\pm\frac{1}{2}\sqrt{k(k+2) \pm 2\sqrt{k(k+2) \pm2\sqrt{ \cdots \pm 2\sqrt{k^2-4}}}},$$
respectively, each of them consisting of $j$ nested radicals in iterative forms. Moreover, the zeros of $f^{n-1}(x)+1$  are
\begin{align*}
&-1,\,k-1,\, \frac{1}{2}(k-2)\pm\frac{1}{2}\sqrt{k^2+4k},\, \frac{1}{2}(k-2)\pm\frac{1}{2}\sqrt{k(k+2) \pm 2\sqrt{k^2+4k}},\ldots, \\
&~~~\frac{1}{2}(k-2)\pm\frac{1}{2}\sqrt{k(k+2) \pm 2\sqrt{k(k+2) \pm2\sqrt{ \cdots \pm 2\sqrt{k^2+4k}}}},
\end{align*}
where the last one consists of $n-2$ nested radicals.
\end{remark}

\section{Laplacian spectrum of $S(n,k)$}\label{sec:S}

 For a graph $\Gamma$, the matrix $L(\Gamma)=D(\Gamma)-A(\Gamma)$ is the {\em Laplacian matrix} of $\Gamma$, where
$D(\Gamma)$ is the diagonal matrix of vertex degrees.
The multi-set of eigenvalues of $L(\Gamma)$ is called the {\em Laplacian spectrum} of $\Gamma$.
In this section, we deal with the Laplacian spectrum of $S(n,k)$.
This is trivial for $n=1$ or $k=1$.
For $k=2,3$, the Laplacian spectrum of $S(n,k)$ is already known (see Remark~\ref{rem:conj} below). We establish the case $n=2$,
and put forward a conjecture  explicitly describing the Laplacian spectrum of $S(n,k)$ in general.

Let $E_{ij}$ be a $k\times k$ matrix in which all entries are $0$, except the $(i,j)$ entry that is $1$.
Consider the $k^2\times k^2$ matrix
$$C:=\sum_{i=1}^k\sum_{j=1}^k(E_{ij}\otimes E_{ji}),$$
where `$\otimes$' denotes the Kronecker product. The matrix $C$ is called  the {\em commutation matrix}. The main property of the commutation matrix (see Magnus and Neudecker \cite{mn}) is that it commutes the Kronecker product:
for any $k\times k$ matrices $M,N$,
$$C(M\otimes N)C=N\otimes M.$$
Note that each row and each column of $C$ corresponds with a pair $(i,j)$ for $1\le i,j\le k$.
Moreover, $C$ is indeed a permutation matrix in which the only $1$ entry in the row $(i,j)$ is located at the column $(j,i)$ for every $1\le i,j\le k$.

For $n=1$, the Laplacian spectrum of $S(1,k)=K_k$ is $\left\{0^{[1]},\,k^{[k-1]}\right\}$, where the superscripts indicate multiplicities.
In the following theorem, we determine the Laplacian spectrum of $S(2,k)$.

\begin{theorem}
The Laplacian spectrum of $S(2,k)$ is the following:
$$\left\{0^{[1]},~ k^{\left[k\choose 2\right]},~(k+2)^{\left[k-1\choose 2\right]},~\left(\frac{1}{2}(k+2)\pm \frac{1}{2}\sqrt{k^2+4}\right)^{[k-1]}\right\}.$$
\end{theorem}
\begin{proof}
First, note that the graph $S(2,k)$ consists of $k$ copies of $K_k$ together with a matching $M$ of size $k\choose2$; exactly one edge for each pair of copies of $K_k$. Let $L$ denote the Laplacian matrix of $S(2,k)$, and $L'$ be the Laplacian matrix of the induced subgraph by the edges of $M$.
It is seen that $L=Q-B$, where $Q=L(kK_k)+I$ and $B=I-L'$. Note that $B$ is a permutation matrix with ${k\choose 2}+k$ eigenvalues $1$ and ${k\choose 2}$ eigenvalues $-1$. Observe that $Q$ has $k$ eigenvalues $1$ and $k^2-k$ eigenvalues $k+1$.
 We have the following bounds on the dimensions of the  intersections of the eigenspaces of $B$ and $Q$:
\begin{align*}
\dim(\E_1(B)\cap\E_{k+1}(Q))&\geq k^2-k+{k\choose 2} +k-k^2={k\choose 2},\\
\dim(\E_{-1}(B)\cap\E_{k+1}(Q))&\geq k^2-k+{k\choose 2}-k^2={k\choose 2}-k,
\end{align*}
in which $\E_\la$ denotes the eigenspace corresponding to the eigenvalue $\la$.
For $\x\in\E_1(B)\cap\E_{k+1}(Q)$, we have $L\x=k\x$ and for $\x\in\E_{-1}(B)\cap\E_{k+1}(Q)$, $L\x=(k+2)\x$.
This means that $L$ has eigenvalues $k$ and $k+2$ with multiplicities at least $k\choose2$ and ${k\choose2}-k$, respectively.

We also have
\begin{equation}\label{eq:Q}Q=I_k\otimes ((k+1)I_k-J_k),\end{equation}
and from the eigenvalues of $Q$,
\begin{equation}\label{eq:Q^2}Q^2-(k+2)Q+(k+1)I=O.\end{equation}
Coming back to $B$, for each of the extreme vertices $(1,1),\ldots,(k,k)$ of $S(2,k)$, there is a $1$ on all the entries of the diagonal of $B$.
The off-diagonal $1$'s correspond with the edges of $M$. By the definition of $S(2,k)$, the edges of $M$ connect the vertices $(i,j)$ and $(j,i)$ for $i\ne j$.
It turns out that $B$ is the commutation matrix, and thus
\begin{equation}\label{eq:BQB}BQB=((k+1)I_k-J_k)\otimes I_k.\end{equation}
The right sides of \eqref{eq:Q} and \eqref{eq:BQB} commute, and so
$$BQBQ=QBQB.$$
Next, we see that
\begin{align*}
 (L^2-(k&+2)L+kI)(QB-BQ)\\&=((Q-B)^2-(k+2)(Q-B)+kI)(QB-BQ)\\
 &=\left(Q^2-(k+2)Q+kI+B^2+(k+2)B-QB-BQ\right)(QB-BQ)\\
  &=\left((k+2)B-QB-BQ\right)(QB-BQ)~~~~\hbox{(by \eqref{eq:Q^2} and since $B^2=I$)}\\
 &=Q^2-(k+2)Q-B\left(Q^2-(k+2)Q\right)B-(QB)^2+(BQ)^2\\
 &=(BQ)^2-(QB)^2\\
 &=O.
\end{align*}

The above equality shows that every vector in the column space of $QB-BQ$ is an eigenvector for $L$ with eigenvalues $\la$, where $\la^2-(k+2)\la+k=0$. To obtain the multiplicity of such $\la$, we compute the rank of $QB-BQ$:
\begin{align}
  \rk(QB-BQ)&=\rk(Q-BQB)\nonumber\\
  &=\rk(I_k\otimes ((k+1)I_k-J_k)-((k+1)I_k-J_k)\otimes I_k)\nonumber\\
  &=\rk(J_k\otimes I_k-I_k\otimes J_k)\nonumber\\
  &=2k-2.\label{eq:2k-2}
\end{align}
To show \eqref{eq:2k-2}, suppose  $P$ is a $k\times k$ matrix whose first column is $\frac1{\sqrt k}(1,\ldots,1)^{\top}$ and that $PP^{\top}=I_k$. Then, $J_k = P(kE_{11})P^{\top}$, and so
$$J_k\otimes I_k-I_k\otimes J_k=(P\otimes P)\big((kE_{11}\otimes I_k)-(I_k\otimes kE_{11})\big)(P^{\top}\otimes P^{\top}).$$
Since $(kE_{11}\otimes I_k)-(I_k \otimes kE_{11})$
 is a diagonal matrix having precisely $2k-2$ non-zero entries in
the columns $2,3,\ldots,k,k+1,2k+1,3k+1,\ldots,(k-1)k+1$,  \eqref{eq:2k-2} follows.

As  $x^2-(k+2)x+k$ is an irreducible polynomial, each of its roots is an  eigenvalue of $L$ with multiplicity at least $k-1$.
The matrix $L$ has a $0$ eigenvalue. Thus, we have obtained so far ${k\choose2}+{k\choose2}-k+2(k-1)+1=k^2-1$ eigenvalues of $L$.
As the sum of the eigenvalues of $L$ is twice the number of edges of $S(2,k)$, it follows that the remaining eigenvalue is $k+2$. So the proof is complete.
\end{proof}

Based on empirical evidence, we put forward the following conjecture.
\begin{conj}\label{conj}\rm
For $n,k\geq 2$, the Laplacian spectrum of $S(n,k)$ consists of the following eigenvalues:
\begin{itemize}
  \item[\rm(i)] $0$ with multiplicity $1$.
    \item[\rm(ii)]  The zeros of $f^j(k-x)$, each with multiplicity $\frac12(k^{n-j}-2k^{n-j-1}+k)$ for $j=0,1,\ldots,n-1$, where $f$ is given in \eqref{eq:f}.
  \item[\rm(iii)] The zeros of $f^j(k-x)+2$, each with multiplicity $\frac12(k^{n-j-1}-1)(k-2)$ for $j=0,1,\ldots,n-2$.
\end{itemize} 
\end{conj}
\begin{remark}\rm The zeros of $f^j(k-x)$ and  $f^j(k-x)+2$
	 for	$j=1$ are $\frac12(k+2\pm\sqrt{k^{2}+4})$ and  $\frac12(k+2\pm\sqrt{k^{2}-4}),$
	respectively, and for $j\ge2$ are of the form
$$\frac{1}{2}(k+2)\pm\frac{1}{2}\sqrt{k(k+2) \pm 2\sqrt{k(k+2) \pm2\sqrt{ \cdots \pm 2\sqrt{k^2+4}}}}$$
and
$$\frac{1}{2}(k+2)\pm\frac{1}{2}\sqrt{k(k+2) \pm 2\sqrt{k(k+2) \pm2\sqrt{ \cdots \pm 2\sqrt{k^2-4}}}},$$
respectively, each of them consisting of $j$ nested radicals.
\end{remark}
\begin{remark}\label{rem:conj}\rm
The graph $S(n,2)$ is the path graph on $2^n$ vertices. Proposition~\ref{prop:path} below shows that Conjecture~\ref{conj} holds for $S(n,2)$.
In Grigorchuk and \v{S}uni\'c \cite{gs}, the spectrum of the Schreier graph $\Gamma_n$ was determined.
The graph $\Gamma_n$ is, in fact, the graph obtained from $S(n,3)$ by adding a loop on each extreme vertex.
By the way the adjacency matrix of $A(\Gamma_n)$ is defined in \cite{gs}, for each loop a $1$ entry on the diagonal is considered, so that each row and column of $A(\Gamma_n)$ has constant sum $3$. It is then observed that the Laplacian spectrum of $S(n,3)$ can be deduced from the spectrum of $\Gamma_n$, which agrees with Conjecture~\ref{conj}.
In summary, Conjecture~\ref{conj} holds for $n=2$ and for $k=2,3$.
\end{remark}

It is known in the literature that the characteristic polynomial of the Laplacian matrix of the paths can be expressed in terms of the 
Chebyshev polynomials. From this fact, for paths with $2^n$ vertices, we obtain the iterated form according to Conjecture~\ref{conj}. For the sake of completeness, we give its complete argument here.
\begin{proposition} \label{prop:path}
The characteristic polynomial of the Laplacian matrix of the path graph on $2^n$ vertices is equal to
$x\prod_{j=0}^{n-1}g^j(2-x)$, where $g(x)=x^2-2$.
\end{proposition}
\begin{proof}  Let $\phi_m$ be the characteristic polynomial of the Laplacian matrix of the path graph on $m$ vertices.
Let $T_m$ and $U_m$ be  the Chebyshev polynomials of degree $m$ of the first and the
second kind, respectively. Then $T_m$ is the only polynomial satisfying
$T_m(\cos\theta)=\cos m\theta$ and $U_m(x) =\sin((m+1)\arccos x)/\sin(\arccos x)$ (Snyder \cite {s}).
From the identities given in Cvetkovi\'c, Doob, and Sachs \cite[p.~220]{cds}, it follows that
$\phi_m(x)=xU_{m-1}(x/2-1)$. By successive use of the identity
$U_{2k-1}(x)=2T_k(x)U_{k-1}(x)$ (see \cite[p.~98]{s}), we get
$$U_{2^n-1}(x)=2^{n-1}T_{2^{n-1}}(x)T_{2^{n-2}}(x)\cdots T_2(x)U_1(x).$$
Note that $U_1(x)=2x$ and $T_2(x)=2x^2-1$. It is seen that $2T_2(x/2-1)=x^2-4x+2=g(2-x)$.
This, together with the identity $T_{2k}(x)=T_2(T_k(x))$, implies that
$2T_{2^j}(x/2-1)=g^j(2-x)$. The proof is now complete.
\end{proof}

\section{What $S^{++}(n,k)$ are Cayley graphs?}\label{sec:Cayley}

Recall that a graph $\Gamma$ is  {\em vertex-transitive} if for any two vertices $u,v$  of $\Gamma$, there exists an automorphism $\sigma$ of $\Gamma$ such that
$\sigma(u)=v$. Let $G$ be a group and $C \subset G$ such that $1\not\in C$ and $c\in C$ implies that $c^{-1}\in C$.
The {\em Cayley graph} $\Cay(G,C)$ with the group $G$ and the `connection set' $C$ is the
graph with vertex set  $G$ in which vertex $u$ is connected to $v$ if and only
if $vu^{-1} \in C$.

It is known that any Cayley graph is vertex-transitive.
 In the other way around, at least for small orders, it seems that the great majority of vertex-transitive graphs are Cayley graphs, see McKay and Praeger \cite{bdm-cep1}. It is expected to continue to be this way for larger orders. In fact, it is conjectured in Praeger, Li, and Niemeyer \cite{cep-chl-acn} that  most vertex-transitive graphs are Cayley graphs.
   In this section, we first determine what $S^{++}(n,k)$ are vertex-transitive and, then, classify $S^{++}(n,k)$ that are Cayley graphs.

\begin{proposition}\label{prop:transitive}
The graph $S^{++}(n,k)$ is vertex-transitive if and only if either $n\le2$ or $k\le2$.
\end{proposition}
\begin{proof}
We have $S^{++}(n,1)\cong K_2$, $S^{++}(1,k)\cong K_{k+1}$, and $S^{++}(n,2)$ is the cycle graph on $2^{n-1}\cdot3$ vertices, which are all vertex-transitive graphs. By Lemma \ref{lineiso}, $S^{++}(2,k)$ is isomorphic to $\li(\s(K_{k+1}))$.
In the graph $\s(K_{k+1})$, the `new' vertices  are in one-to-one correspondence with $2$-subsets of $[k+1]$.
It is then easy to see that that any permutation of $[k+1]$ induces an automorphism of  $\s(K_{k+1})$. Now,
for a given pair of edges of $\s(K_{k+1})$ which can be represented as $e=\{i,\{i,j\}\}$ and $e'=\{i',\{i',j'\}\}$,
the automorphism induced by a permutation $\sigma$ that $\sigma(i)=i'$ and $\sigma(j)=j'$, maps $e$ to $e'$.
It follows that   $\s(K_{k+1})$ is edge-transitive, and so $\li(\s(K_{k+1}))\cong S^{++}(2,k)$ is vertex-transitive.  Hence, assume that $n\geq3$ and $k\ge3$.
  We observe that an extreme vertex of a copy $\Delta$ of $S(n-1,k)$ in $\Gamma=S^{++}(n,k)$ cannot be mapped to a non-extreme vertex of  $\Delta$ by any automorphism of $\Gamma$.
    To be more precise, let $\uu=(1,\ldots,1,1)$ and $\vv=(1,\ldots,1,2)$.
    It can be seen that $\uu$ is a cut vertex for the induced subgraph by the vertices at distance at most $3$ form $\uu$,
    while $\vv$ is not a cut vertex for the induced subgraph by the vertices at distance at most $3$ form $\vv$.
    It follows that $\uu$ cannot be mapped to $\vv$ by any automorphism of $\Gamma$, and thus $\Gamma$ is not vertex-transitive.
\end{proof}

From Proposition~\ref{prop:transitive}, it follows that the graphs $S^{++}(n,k)$ for $n\ge3$ and $k\ge3$ cannot be Cayley graphs.
The graphs $S^{++}(n,1)$, $S^{++}(n,2)$, and $S^{++}(1,k)$ are all Cayley graphs. It remains to characterize what $S^{++}(2,k)$ are Cayley graphs for $k\ge3$.
This is our goal in the rest of this section.

\begin{definition}\em
Let $\Gamma$ be a graph and $\Delta$ a subgraph of $\Gamma$.
We say that $\Gamma$ is  {\em strongly $\Delta$-partitioned} if:
\begin{itemize}
\item[\rm(i)] The vertex set of $\Gamma$ is partitioned by the vertex sets of copies $\Delta_0,\ldots,\Delta_k$ of $\Delta$.
\item[\rm(ii)] Apart from $\Delta_0,\ldots,\Delta_k$, the graph $\Gamma$ contains no further copies of $\Delta$.
\end{itemize}
\end{definition}

By the way $S^{++}(n,k)$ is defined, it is constructed based on $k+1$ copies of $S(n-1,k)$.
The following proposition gives a structural property of $S^{++}(n,k)$
that it is indeed strongly $S(n-1,k)$-partitioned for $n\ge2$ and $k\ge3$.
Note that this is not the case for $k=2$ because $S^{++}(n,2)$, that is a cycle with $3\cdot2^{n-1}$ vertices, contains more than three copies of $S(n-1,2)$, which is a path on  $2^{n-1}$ vertices. Although we only need the case $n=2$ of the proposition, we state it in its full generality because it could be of independent interest.

\begin{proposition}\label{vertex partition}
Let $n\ge2$ and $k\ge3$. The graph $S^{++}(n,k)$ is strongly $S(n-1,k)$-partitioned.
\end{proposition}
\begin{proof}
Let $\Gamma:=S^{++}(n,k)$ and $\Gamma_0,\ldots,\Gamma_k$ be the $k+1$ copies of $S(n-1,k)$ used to construct $\Gamma$ by its definition.
Clearly, $V(\Gamma_0),\ldots,V(\Gamma_k)$ is a partition of $V(\Gamma)$.
We show that $\Gamma$ contains no more copies of $S(n-1,k)$. Let $\Delta$ be a subgraph of $\Gamma$ isomorphic to $S(n-1,k)$.

First, assume that $n=2$. Let $u\in V(\Delta)\cap V(\Gamma_t)$ for some $t$, with $0\leq t\leq k$. Since $u$ has at most one neighbor in $V(\Gamma)\setminus V(\Gamma_t)$ and $k\geq3$, there  exists another vertex $v\in V(\Delta)\cap V(\Gamma_t)$ adjacent to $u$. Now, if $w$ is any vertex of $\Delta$, then since $w$ is adjacent to two vertices $u$ and $v$ of $\Gamma_t$ it must belong to $V(\Gamma_t)$. Hence $V(\Delta)\subseteq V(\Gamma_t)$ and, consequently, $\Delta=\Gamma_t$.

Now, let $n\geq3$. Note that $S(n-1,k)$ is connected and has no bridges since every edge of $S(n-1,k)$ lies on a cycle (which can be seen by induction on $n$). If $\Delta\ne\Gamma_i$ for $i=0,\ldots,k$, then $\Delta$ shares its vertices with at least two $\Gamma_s$ and $\Gamma_t$.
By the definition, exactly one extreme vertex, say $u$, of $\Gamma_s$ is adjacent to exactly one extreme vertex, say $v$, of $\Gamma_t$.
Because of the connectivity, $\Delta$ must contain the edge $uv$. Note that
for any vertex $w$ outside $\Gamma_s$ and $\Gamma_t$, the distance between $w$ and either $u$ or $v$ is greater than the diameter of $S(n-1,k)$, and so
$w\not\in V(\Delta)$. It follows that $\Delta$ is a subgraph of $\Gamma':=\Gamma[V(\Gamma_s)\cup V(\Gamma_t)]$.
However, $uv$ is a bridge for $\Gamma'$ and thus a bridge for $\Delta$, a contradiction.
\end{proof}

The following lemma reveals the structure of strongly $\Delta$-partitioned Cayley graphs.
\begin{lemma}\label{Cayley partitioned graph: induced subgraph=coset}
Let $\Gamma$ be a Cayley graph with  a subgraph $\Delta$ such that $\Gamma$ is strongly $\Delta$-partitioned. Then, the vertex sets of the copies of $\Delta$ are all the right cosets of a subgroup of the underlying group of $\Gamma$.
\end{lemma}
\begin{proof}
Let $\Gamma$ be a Cayley graph on a group $G$, and $X\subseteq G$ be such that $1\in X$ and $\Gamma[X]$, the subgraph of $\Gamma$ induced by $X$, is isomorphic to $\Delta$. Since for any $x\in X$, $\Gamma[Xx^{-1}]$ is isomorphic to $\Delta$ and $1\in Xx^{-1}$, from the hypothesis of the lemma, it follows that $Xx^{-1}=X$. Thus $XX^{-1}=X$ and, hence, $X$ is a subgroup of $G$. As for any $g\in G$, $\Gamma[Xg]$ is isomorphic to $\Delta$ and the sets $Xg$ cover all elements of $G$, it follows that every induced subgraph of $\Gamma$ isomorphic to $\Delta$ is a right coset of $X$, as required.
\end{proof}

\begin{definition}\em
Let $\Gamma$ be a strongly $\Delta$-partitioned graph. We say that $\Gamma$ has {\em connection constant} $c$ if there are exactly $c$ edges between any two copies of $\Delta$ in $\Gamma$. We denote the set of all strongly $\Delta$-partitioned graphs with connection constant $c$ by $\SP_c(\Delta)$.
\end{definition}

\begin{remark}\rm  The family $\SP_1(K_d)$ contains an only regular graph. However, this is not the case in any regular graph $\Delta$.
 If $\Gamma\in\SP_1(\Delta)$ is a regular graph with $\Delta$ being a $d$-regular graph on $k$ vertices, then $\Gamma$ necessarily contains $k+1$ copies of $\Delta$ and thus $\Gamma$ is $(d+1)$-regular with $k(k+1)$ vertices. For instance, in the case in which $\Delta$ is $C_4$, the cycle on $4$ vertices, $\Gamma$ is a cubic graph on $20$ vertices.
By a computer search, we found all the regular graphs in $\SP_1(C_4)$. It turned out that there are seven non-isomorphic such graphs, among which only one is a Cayley graph.
\end{remark}

Here we recall some notions from group theory that will be used in what follows. Let $G$ be a finite group and $H$ be a nontrivial proper subgroup of $G$. The conjugate of $H$ by an element $g$ of $G$ is defined as $H^g=\{h^g:h\in G\}$, where $h^g:=g^{-1}hg$ denotes the conjugate of $h$ by $g$. The group $G$ is called a \textit{Frobenius group} with \textit{Frobenius complement} $H$ if $H\cap H^g=\{1\}$ for all $g\in G\setminus H$. A celebrated theorem of Frobenius states that  $N:=G\setminus\bigcup_{g\in G}(H\setminus\{1\})^g$ is a normal subgroup of $G$, called the Frobenius kernel of $G$, satisfying $G=NH$ and $N\cap H=\{1\}$, that is, $G=N\rtimes H$ is a semidirect product of $N$ by $H$ (see \cite[8.5.5]{djsr}). The other concepts we use in the following are standard and can be found in Robinson \cite{djsr}.
\begin{theorem}\label{Cayley pseudo-Sierpinski graph with c=1}
Suppose that $\Gamma$ and $\Delta$ are two regular graphs and $\Gamma\in\SP_1(\Delta)$.
If $\Gamma$ is a Cayley graph $\Cay(G,C)$, then $|\Delta|+1=p^m$ is a prime power, $G=N\rtimes H$ is a Frobenius group with minimal normal Frobenius kernel $N\cong\Bbb{Z}_p^m$ and Frobenius complement $H$, $C=C'\cup\{c\}$ with $\Delta\cong\Cay(H,C')$ and $c^2=1$, and either
\begin{itemize}
\item[\rm(i)]$c\in N$ and $H=\gen{C'}$, or
\item[\rm(ii)]$c=h^n$ for some $h\in H\setminus\{1\}$ and $n\in N\setminus\{1\}$, and $H=\gen{C',h}$.
\end{itemize}
Conversely, if $\Delta$ satisfies the above conditions, then $\Cay(G,C)\in\SP_1(\Delta)$.
\end{theorem}
\begin{proof}
Let $\Gamma=\Cay(G,C)$, and $\Delta_0,\Delta_1,\ldots,\Delta_k$ be the copies of $\Delta$ in $\Gamma$.
As there is exactly one edge between any two copies of $\Delta$, it is observed that $|\Delta|=k$ and $\Gamma$ is $(d+1)$-regular if $\Delta$ is $d$-regular. Let $H:=V(\Delta_0)$ and assume, without loss of generality, that $1\in H$. By Lemma \ref{Cayley partitioned graph: induced subgraph=coset}, $H$ is a subgroup of $G$. Let $C'$ be the neighborhood of $1$ in $\Gamma[H]$. Since $\Gamma$ is $(d+1)$-regular, besides the elements of $C'$, the vertex $1$ has exactly one other neighbor, say $c\in G\setminus H$. So $C=C'\cup\{c\}$. Since $H$ is a subgroup of $G$, $C'^{-1}\subseteq H$, which implies that $C'^{-1}=C'$. Thus, $c=c^{-1}$ is an involution. Clearly, $Hc\neq H$ so that $\Gamma[Hc]=\Delta_i$ for some $1\leq i\leq k$. On the other hand,
\[1=|E(\Delta_0,\Delta_i)|=|\{\{h,ch\}:h\in H\cap H^c\}|=|H\cap H^c|,\]
from which it follows that $H\cap H^c=\{1\}$. Now, a simple verification shows that $Hch\cap Hch'=\emptyset$ for all distinct elements $h,h'\in H$. Since $\Gamma[H]=\Delta_0$ and $\Gamma[Hch]$ ($h\in H$) are equal to $\Delta_1,\ldots,\Delta_k$ in some order, we must have
\[G=H\cup\bigcup_{h\in H}Hch,\]
where the unions are disjoint. As a result, every element $g\in G\setminus H$ can be written as $g=hch'$ for some $h,h'\in H$, from which it follows that
\[H\cap H^g=(H^{{h'}^{-1}}\cap (H^h)^c)^{h'}=(H\cap H^c)^{h'}=\{1\}^{h'}=\{1\}.\]
Hence, $G$ is a Frobenius group with complement $H$. Let $N$ be the Frobenius kernel of $G$. By \cite[10.5.1(i)]{djsr}, $N$ is nilpotent. Let $N_0$ be a nontrivial characteristic subgroup of $N$ with minimum order. Then $N_0$ is a normal subgroup of $G$ (see \cite[1.5.6(iii)]{djsr}). Note that $N_0$ is an elementary Abelian $p$-group for $N_0$ is nilpotent and the subgroup of $N_0$ generated by central elements of a given prime order $p$ dividing $|Z(N_0)|$ is a characteristic subgroup of $N_0$ and hence of $N$ (see \cite[1.5.6(ii)]{djsr}). If $N\neq N_0$, then $N_0H$ is a Frobenius group for $N_0H$ is a subgroup of $G$ and $H\cap H^g=1$ for all $g\in N_0H\setminus H$. Moreover, as a proper subgroup of $N$, $|N_0|\leq|N|/2\leq(k+1)/2$ and hence $|N_0|-1$ is not divisible by $|H|=k$ contradicting \cite[Exercises 8.5(6)]{djsr}. Thus $N=N_0$ so that $k+1=|N|=p^m$ is a prime power for some $m\geq1$. Note that $N$ is a minimal normal subgroup of $G$ for if $N$ contains a nontrivial normal subgroup $N_0$ of $G$ properly, then $N_0H$ would be a Frobenius group which leads us to the same contradiction as above. If $c\in N$, then since $G\subseteq N\gen{C'}$ it follows that $H=\gen{C'}$. Now assume that $c\notin N$. Then $c^n\in H\setminus\{1\}$ for some $n\in N\setminus\{1\}$. As $G\subseteq N\gen{C',c^n}$ it follows that $H=\gen{C',c^n}$, as required. The converse is straightforward.
\end{proof}

We are now in a position to conclude the main result of this section.
\begin{theorem}\label{Cayley S++(n,k)}
The graph $S^{++}(n,k)$  is a Cayley graph if and only if either
\begin{itemize}
\item[\rm(i)] $n=1$, 
\item[\rm(ii)] $k\le2$, or 
\item[\rm(iii)] $n=2$ and $k+1=p^m$ is a prime power. 
\end{itemize}
Furthermore, in the case {\rm(iii)}, we have
\[S^{++}(n,k)\cong\Cay(G,(H\setminus\{1\})\cup\{c\}),\]
for every Frobenius group $G$ with complement $H$ of order $p^m-1$, elementary Abelian minimal normal Frobenius kernel of order $p^m$, and involution $c\in G\setminus H$.
\end{theorem}
\begin{proof}
By Proposition~\ref{prop:transitive}, $S^{++}(n,k)$ for $n\ge3$ and $k\ge3$ is not a Cayley graph.
As mentioned above, $S^{++}(n,1)$, $S^{++}(n,2)$, and $S^{++}(1,k)$ are all Cayley graphs.
So, we may assume that $n=2$ and $k\ge3$.

First, we show that $S^{++}(2,q-1)$ are Cayley graphs for all prime powers $q$. Let $\Bbb{F}_q$ denote the finite field with $q$ elements. Then $G:=\Bbb{F}^*_q\times\Bbb{F}_q$ together with the multiplication
\[(x,a)\cdot(y,b)=(xy,xb+a),\]
forms a group known as {\em one dimensional affine group}. We show that $S^{++}(2,q-1)\cong\Cay(G,C)$, where $C=\left\{(x,0): 1\ne x\in\Bbb{F}^*_q\right\}\cup\{(-1,-1)\}$. To this end, let $H:=\{(x,0): x\in\Bbb{F}^*_q\}$ be a subgroup of $G$ of order $q-1$. Then $H$ has $q$ right cosets each of which induces a complete subgraph in $\Cay(G,C)$ for $h'g(hg)^{-1}=h'h^{-1}\in C$ for all distinct elements $hg$ and $h'g$ of a right coset $Hg$ of $H$. Since $(x,0)(1,ax^{-1})=(x,a)$ covers all elements of $G$ when $x$ and $a$ ranges over $\Bbb{F}^*_q$ and $\Bbb{F}_q$, respectively, it follows that every right coset of $H$ has a representative of the form $(1,b)$ for some $b\in \Bbb{F}_q$. Let $Hg$ and $Hg'$ be distinct right cosets of $H$ with $g=(1,a)$ and $g'=(1,a')$. Then an element $hg$ of $Hg$ is adjacent to an element $h'g'$ of $Hg'$ if and only if $h'g'g^{-1}h^{-1}=(h'g')(hg)^{-1}\in C$ or equivalently $(h'g')(hg)^{-1}=(-1,-1)$ as $g'g^{-1}\notin H$. A simple verification shows that this equation has a unique solution for $(h,h')$ so that there is a unique edge between any two right cosets of $H$. Indeed, $h=(x,0)$ and $h'=(x',0)$ satisfy the equation if and only if $-x'=x=(a'-a)^{-1}$. Hence, from the definition, it follows that $S^{++}(2,q-1)\cong\Cay(G,C)$.

Now, assume that $\Gamma:=S^{++}(2,k)\cong\Cay(G,C)$ be a presentation of $S^{++}(2,k)$ as a Cayley graph. By Proposition~\ref{vertex partition}, $\Gamma$ is strongly $\Gamma_0$-partitioned for some complete subgraph $\Gamma_0$ of $\Gamma$ of order $k$. Let $H:=V(\Gamma_0)$ and assume that $1\in H$. We know from Lemma \ref{Cayley partitioned graph: induced subgraph=coset} that $H$ is a subgroup of $G$.
By Theorem~\ref{Cayley pseudo-Sierpinski graph with c=1}, $k+1=p^m$ is a prime power, $G=N\rtimes H$ is a Frobenius group with Frobenius kernel $N$ and Frobenius complement $H$ such that $N\cong \Bbb{Z}_p^m$ is a minimal normal subgroup of $G$, $C=C'\cup\{c\}$, $C'^{-1}=C'\subseteq H$, $c^2=1$, and either
\begin{itemize}
\item[\rm(a)]$c\in N$ and $H=\gen{C'}$, or
\item[\rm(b)]$c=h^n$ for some $h\in H\setminus\{1\}$ and $n\in N\setminus\{1\}$, and $H=\gen{C',h}$.
\end{itemize}
Since $\Gamma_0$ is a complete graph, we must have $H\setminus\{1\}\subseteq C$. Then, (a) and (b) together are equivalent to say that $c\in G\setminus H$. The proof is now complete.
\end{proof}

As a generalization of Theorem~\ref{Cayley pseudo-Sierpinski graph with c=1}, we pose the following problem.

\begin{problem}\rm
Let $\Delta$ be a regular graph. Classify all Cayley graphs in $\SP_c(\Delta)$ for $c\ge2$.
\end{problem}

\section{New non-Cayley numbers}\label{sec:NC}

In this final section, we give a partial answer to a famous rather old open problem in algebraic graph theory. A positive integer $n$ is called a \textit{Cayley number} if all vertex-transitive graphs of order $n$ are Cayley graphs. Maru\v{s}i\v{c} \cite{dm} in 1983 posed the problem of characterizing the set $\NC$ of all non-Cayley numbers. Since disjoint unions of copies of vertex-transitive (non-Cayley) graphs are again vertex-transitive (non-Cayley) graphs, it follows that every multiple of a non-Cayley number is again a non-Cayley number.
Hence the problem of determining $\NC$ reduces to finding `minimal' non-Cayley numbers. It is well-known that all primes are Cayley numbers. Following a series of papers by various authors, McKay and Praeger \cite{bdm-cep2} and Iranmanesh and Praeger \cite{mai-cep} provided necessary and sufficient conditions under which the product of two and three distinct primes is a Cayley number, respectively. In the same paper, McKay and Praeger established the following remarkable result determining all non-square-free Cayley numbers.
\begin{theorem}[McKay and Praeger \cite{bdm-cep2}]
Let $n$ be a positive integer that is divisible by the square of a prime $p$. Then $n\in\NC$ unless $n=p^2$, $n=p^3$, or $n=12$.
\end{theorem}

It follows that, for determining $\NC$, it is enough to consider only square-free positive integers. While the problem is yet open for the products of at least four distinct primes, there are partial results worth to mention here.
\begin{theorem}[Dobson and Spiga \cite{td-ps}]
There exists an infinite set of primes every finite product of its distinct elements is a Cayley number.
\end{theorem}

As a consequence of Theorem~\ref{Cayley S++(n,k)}, the graph $S^{++}(2,k)$ that has $k(k+1)$ vertices is not a Cayley graph if $k+1$ is not a prime power.
Therefore, we obtain a new infinite class of square-free non-Cayley numbers as follows.
\begin{theorem}\label{non-Cyaley numbers}
Let $k$ be any positive integer such that $k(k+1)$ is square-free, and $k+1$ is not a prime. Then, $k(k+1)\in\NC$.
\end{theorem}

As mentioned in Dobson and Spiga \cite{td-ps}, it is straightforward by making use of the group-theoretic and the number-theoretic results already available in the literature to prove that Cayley numbers have density zero in the natural numbers, and hence the density
of non-Cayley numbers is $1$. In the light of this  fact, one might wonder about the distribution of the numbers $k$ satisfying the conditions of Theorem~\ref{non-Cyaley numbers} in the set of positive integers.
The following theorem shows that for large enough $N$, more than one third of positive integers less than or equal to $N$  satisfies the conditions of
Theorem~\ref{non-Cyaley numbers}.
\begin{theorem}
The density of the set
\[\{k :  k(k+1)~\hbox{is square-free, and $k+1$ is not a prime}\}\]
is about $0.3226$
\end{theorem}
\begin{proof}
Let $f\in\Bbb{Z}[t]$ be a primitive polynomial (that is, the greatest common divisor of its coefficients is 1) without multiple roots such that its image on $\Bbb{N}$ has $k$-free greatest common divisor. Recall that a number that is not divisible by any proper $k$-th power is called $k$-free. Let $\mathcal{S}_f^k(x)$ denote the number of all positive integers $n\leq x$ such that $f(n)$ is $k$-free, and consider
\[\delta_{f,k}:=\prod_{p\ \text{prime}}\left(1-\frac{\varrho(p^k)}{p^k}\right),\]
where $\varrho(d)$ denotes the number of roots of $f$ in $\Bbb{Z}_d$. Ricci \cite{gr} (see also Pappalardi \cite{fp}) proved that
\[\mathcal{S}_f^k(x)\sim\delta_{f,k}x\]
provided that $\deg f\leq k$. Clearly, the function $f(t):=t(t+1)$ satisfies the above requirements of Ricci's theorem for $k=2$. Also, it is obvious that $\varrho(p^2)=2$ for all primes $p$. Thus, by Ricci's theorem, the density of all positive integers $k$, for which $k(k+1)$ is square-free, in the set of all positive integers, is equal to
\[\delta_{f,2}=\prod_{p\ \text{prime}}\left(1-\frac{2}{p^2}\right)=2C_{\text{Feller-Tornier}}-1\approx0.3226340989,\]
where $C_{\text{Feller-Tornier}}$ is the Feller-Tornier constant (see Finch \cite[\S 2.4.1]{srf}). Since primes have zero density in the set of all positive integers, the result follows.
\end{proof}

To date, all the numbers $n$ whose membership in $\NC$ is known are determined based on the results of \cite{mai-cep, bdm-cep1, bdm-cep2}.
Using a computer search, we see that the list of the numbers whose membership in $\NC$ are not yet determined begins with
$$9982,\ 12958,\ 18998,\ 19646,\ 20398,\ 21574,\ 24662,\ 25438,\ 25606,\ \ldots.$$
A simple computation reveals that among the numbers leass than or equal to $10^8$, there are $2763$ square-free integers of the form $k(k+1)$, with $k+1$ not a prime of which the following eight integers are new non-Cayley numbers:
$$1386506,\ 2668322,\ 15503906,\ 23985506,\ 38359442,\ 74261306,\ 89898842,\ 95912642.$$

\section*{Acknowledgments}
The authors would like to thank anonymous referees for constructive comments which led to improvement of the presentation of the paper. The second author carried this work during a Humboldt Research Fellowship at the University of Hamburg. He thanks the Alexander von Humboldt-Stiftung for financial support. 

\end{document}